\documentclass[12pt]{article}
\usepackage[cp1251]{inputenc}
\usepackage[english, russian]{babel}
\usepackage{amsfonts,amssymb,mathrsfs,amscd}

\usepackage{amsthm}
\usepackage{amsmath}

\usepackage{hyperref}
\usepackage{cleveref}

\numberwithin{equation}{section}

\title{Инвариантные подпространства
оператора обобщенного обратного сдвига и рациональные функции}

\medskip

\author{ О.~А.~Иванова, С.~Н.~Мелихов, Ю.~Н.~Мелихов}

\date{}

\newtheorem{theorem}[]{Теорема}
\newtheorem{corollary}[]{Следствие}
\newtheorem{lemma}{Лемма}

\theoremstyle{definition}
\newtheorem{definition}[]{Определение}
\newtheorem{remark}[]{Замечание}

\begin{document}

\maketitle \thispagestyle{empty}

\begin{abstract}
Приводится полная характеризация собственных замкнутых
инвариантных подпространств оператора обобщенного обратного сдвига
(оператора Поммье) в пространстве Фреше всех функций, голоморфных
в односвязной области $\Omega$ комплексной плоскости, содержащей
начало. В случае, когда порождающая этот оператор функция не имеет
нулей в $\Omega$,  все такие подпространства являются
конечномерными. Если дополнительно $\Omega$ совпадает со всей
комплексной плоскостью, то рассматриваемый оператор обобщенного
обратного сдвига является одноклеточным. Если эта функция имеет
нули в $\Omega$, то семейство упомянутых инвариантных
подпространств распадается на два клаcса: первый состоит из
конечномерных подпространств, а второй --- из бесконечномерных.

\medskip

{\bf Ключевые слова: } Инвариантное подпространство, пространство
голоморфных функций, оператор обратного сдвига
\end{abstract}

\section*{Введение}
Пусть $\Omega$ --- область в $\mathbb C$, содержащая начало; $H(\Omega)$ ---
пространство всех голоморфных в $\Omega$ функций с топологией компактной сходимости. Фиксированная функция
$g_0\in H(\Omega)$ такая, что $g_0(0)=1$, задает оператор обобщенного обратного сдвига (оператор Поммье)
$D_{0,g_0}(f)(t):=\frac{f(t)-g_0(t)f(0)}{t}$, линейно и непрерывно действующий в $H(\Omega)$.
Если $g_0\equiv 1$, то $D_{0,g_0}$ --- обычный оператор $D_0$ обратного сдвига.
(В общем случае $D_{0,g_0}$ является одномерным возмущением
$D_0$.) Как отметил Ю.~С. Линчук \cite{LINCHUK}, линейный
непрерывный оператор в $H(\Omega)$ является левым обратным к
оператору умножения на независимую переменную тогда и только
тогда, когда он совпадает с некоторым оператором $D_{0,g_0}$.

В настоящей работе изучаются собственные замкнутые
$D_{0,g_0}$-инвариант\-ные подпространства $H(\Omega)$.
Исследование инвариантных подпространств линейных непрерывных
операторов в локально выпуклых пространствах имеет большое
значение как вследствие внутренних потребностей функционального
анализа и теории функций, так и в связи с многочисленными
приложениями (см. обзор Н.~К. Никольского \cite{NIK}). В настоящее
время имеется довольно обширная литература, посвященная
циклическим векторам оператора $D_0$ в пространствах голоморфных
функций и инвариантным подпространствам $D_0$ в банаховых
пространствах голоморфных функций. При этом элемент $x$ локально
выпуклого пространства $E$ называется циклическим вектором
линейного непрерывного оператора $A$ в $E$, если орбита $\{
A^n(x)\,|\, n\in\mathbb N\cup\{0\}\}$ полна в $E$, т.\,е.
замыкание ее линейной оболочки в $E$ совпадает с $E$.
Непосредственная связь цикличности и инвариантных подпространств
заключается в следующем: элемент $x\in E$ --- циклический вектор
$A$ в $E$ тогда и только тогда, когда $x$ не принадлежит ни одному
собственному замкнутому $A$-инвариантному подпространству $E$.

Одним из первых исследований в этом направлении является,
по-види\-мому, статья Р. Дугласа, Г. Шапиро, А. Шилдса
\cite{DSHSH}, в которой была решена проблема Д. Сарасона
характеризации циклических векторов и инвариантных подпространств
оператора $D_0$ в пространстве Харди $H^2$ в единичном круге
$\mathbb D:=\{z\in\mathbb C\,|\,|z|<1\}$. В \cite{DSHSH} была
обнаружена связь $D_0$-нецикличности функций из $H^2$ с
возможностью их псевдопродолжения через граничную окружность до
некоторой мероморфной функции в $\{ z\in\mathbb C\,|\, |z|>1\}$
\cite[теорема 2.2.1]{DSHSH}. Работа \cite{DSHSH} послужила толчком
к интенсивному развитию соответствующей теории. К исследованию,
предпринятому в настоящей статье, непосредственное отношение имеют
связи циклических векторов и инвариантных подпространств $D_0$ с
рациональной аппроксимацией (см., например, \cite{TUM1},
\cite[\S~4]{DSHSH}, \cite{GRIBOV}). В связи с этим отметим
следующий факт, вытекающий из результатов Г.~Ц. Тумаркина и
объединяющий банахов и небанахов случаи: функция $f$, голоморная в
некотором круге $|z|<R$, $R>1$, не является циклическим вектором
$D_0$ в $H^2$ в том и только в том случае, когда $f$ является
рациональной дробью. Для функции $f\in H(\Omega)$ ее
$D_0$-нецикличность в $H(\Omega)$ равносильна тому, что $f$ ---
рациональная дробь. Из работ в данном направлении, относящихся к
пространствам, отличным от $H^2$, отметим монографию Й. Цимы и У.
Росса \cite{CIRO}, в которой исследованы циклические векторы $D_0$
в пространстве $H^p$ для показателей $p\in(0,\infty)$. А. Алеман, С. Рихтер, К. Сандберг \cite{ARS} исследовали (в терминах псевдопродолжения
в $|z|>1$) свойства элементов собственных замкнутых инвариантных подпространств оператора обратного
сдвига в гильбертовых пространствах голоморфных функций с регулярной нормой. Примерами таких
пространств являются пространства Харди в $\mathbb D$, весовые
пространства Бергмана. По поводу других работ в этом направлении
см. библиографию в \cite{CIRO}, \cite{ARS}. В статье
\cite{ITOGI17} были охарактеризованы собственные замкнутые
$D_{0,g_0}$-инвариантные подпространства пространства целых
функций экспоненциального типа, реализующего посредством
преобразования Фурье-Лапласа сильное сопряженное к пространству
ростков функций, голоморфных на выпуклом локально замкнутом
подмножестве $\mathbb C$.

Оператор $D_{0,g_0}$, рассматриваемый здесь, действует в
пространстве Фреше $H(\Omega)$, которое ненормируемо. Мягкость
топологии в $H(\Omega)$ приводит к наличию  широкого множества
циклических векторов. Их описание  более элементарно, чем в
известных банаховых случаях, а граничное поведение нециклических
функций $f\in H(\Omega)$ (если $g_0$ не имеет нулей в $\Omega$)
жесткое: нецикличность $f$ равносильна голоморфному продолжению
$f/g_0$ до мероморфной функции в расширенной комплексной
плоскости, т.\,е. $f/g_0$ должна быть рациональной функцией. Это
приводит к "более алгебраическому" \, описанию соответствующих
циклических векторов и инвариантных подпространств в $H(\Omega)$.
Характеризация циклических векторов оператора $D_0$ в $H(\Omega)$
для односвязной области $\Omega$ была получена Ю.~А. Казьминым
\cite{KAZ}, для конечно-связных областей $\Omega$ --- Н.~Е. Линчук
\cite{LINCHUKNE}. В случае, когда $\Omega$ --- круг, ранее она
была дана М.~Г. Хаплановым \cite{KHAPL}. Ю.~С. Линчук
\cite{LINCHUK} описал циклические векторы $D_{0,g_0}$ в
$H(\Omega)$ для односвязной области $\Omega$ при предположении,
что $g_0$ не имеет нулей в $\Omega$. Это ограничение снято в
\cite{UFA15}. Заметим, что Ж. Годфруа и Дж. Шапиро \cite{GOSHA}
исследовали более сильное, чем цикличность, динамическое свойство
-- гиперцикличность -- линейного непрерывного оператора в
банаховом пространстве, названного ими также {\it обобщенным
обратным сдвигом}. (По поводу используемой терминологии см.
замечание \ref{GOSHA}.)

Главные результаты статьи --- теоремы 2--4, описывающие, соответственно, случаи, когда
$g_0$ не имеет нулей в $\Omega\ne\mathbb C$, имеет нули в $\Omega\ne\mathbb C$
(при этом область $\Omega$ односвязна)
и $\Omega=\mathbb C$.
Оказалось, что ситуации, когда $g_0$ не имеет нулей и имеет нули в $\Omega$, и похожи, и существенно отличаются.
В первой из них все собственные замкнутые $D_{0,g_0}$-инвариантные подпространства конечномерны
и задаются конечным числов полюсов вне $\Omega$ и порядками соответствующих рациональных функций,
а во второй к ним добавляются и бесконечномерные подпространства, определяемые нулями функции $g_0$.
Оператор $D_{0,g_0}$ является одноклеточным в $H(\Omega)$ тогда и только тогда, когда
$\Omega=\mathbb C$ и $g_0$ не имеет нулей в $\mathbb C$.

Основных методов, используемых в данной работе, два. Первый важен в случае, когда $g_0$ имеет нули и заключается в
использовании "экстремальных" \, функций из $D_{0,g_0}$-инвариантных подпространств,
обращающихся в нуль на заданном подмножестве нулевого множества $g_0$ с заданными кратностями
и не равных нулю на остальной части нулевого множества $g_0$. Второй --- естественно возникающий
метод "просеивания" \,
рациональных функций с помощью многочленов от оператора обратного сдвига.

\section{Предварительные сведения. Вспомогательные результаты}

Пусть $g_0$ --- голоморфная в начале функция такая, что $g_0(0)=1$.
Она задает оператор обобщенного обратного сдвига (оператор Поммье) $D_{0, g_0}$
в пространстве ростков функций, голоморфных в начале:
$$
D_{0,g_0}(f)(t):=\begin{cases}
\frac{f(t)-g_0(t)f(0)}{t}, & \text{$t\ne 0$,}\\
f'(0)-g_0'(0)f(0),& \text{$t=0$.}
\end{cases}
$$
Для произвольной области $\Omega$ в $\mathbb C$, содержащей начало,
оператор $D_{0,g_0}$
линейно и непрерывно действует в пространстве $H(\Omega)$ всех голоморфных в $\Omega$ функций с топологией
компактной сходимости.

Ниже существенно используется описание циклических векторов $D_{0,g_0}$ в $H(\Omega)$.
При этом элемент $x$ локально выпуклого пространства $E$ называется циклическим вектором
линейного непрерывного оператора $A: E\to E$, если замыкание линейной оболочки
орбиты $\{A^n(x)\,|\, n\ge 0\}$ совпадает с $E$.
Приведем результат из \cite{UFA15}.

\begin{theorem}\label{cycl}
Пусть $\Omega$ --- односвязная область в $\mathbb C$, содержащая начало.
Для $f\in H(\Omega)$ следующие утверждения равносильны:

\begin{itemize}
\item[(i)] $f$ --- циклический вектор $D_{0,g_0}$ в $H(\Omega)$.
\item[(ii)] Функции $f$ и $g_0$ не имеют общих нулей в $\Omega$ и не
существует рациональной функции $R$ такой, что $f=Rg_0$.
\end{itemize}
\end{theorem}

Эта теорема применяется далее в контексте следующего утверждения:
собственное замкнутое $D_{0,g_0}$-инвариантное подпространство $S$ пространства $H(\Omega)$ не содержит ни одного циклического вектора $D_{0,g_0}$ в
$H(\Omega)$.

\medskip
Далее до конца \S\,1 \, $\Omega$ --- область в $\mathbb C$,
содержащая начало. Для $h\in H(\Omega)$, $Q\subset H(\Omega)$
полагаем $hQ:=\{hf\,|\, f\in Q\}$. Пусть $\mathbb C[z]$ --- кольцо
всех многочленов над полем $\mathbb C$, $\mathbb C[z]_n$, $n\ge
0$, --- множество всех многочленов над $\mathbb C$ степени не выше
$n$, $\mathbb C(z)$ --- кольцо всех рациональных функций над полем
$\mathbb C$; $\mathcal Z_0$ --- множество всех функций из
$H(\Omega)$, имеющих хотя бы один общий нуль с $g_0$; $\mathbb
C^-_\Omega(z)$ --- множество всех правильных рациональных дробей,
голоморфных в $\Omega$. Если $\Omega=\mathbb C$, то $\mathbb
C^-_\Omega(z)=\{0\}$. Из теоремы 1 вытекает, что все собственные
замкнутые $D_{0,g_0}$-инвариантные подпространства $H(\Omega)$
содержатся в $\mathcal Z_0\cup \left(g_0\mathbb C(z)\right)$.

Отметим, что $\mathbb C^-_\Omega(z)={\rm
span}\left\{q_{\lambda,n}\,|\,\lambda\in\mathbb C\backslash\Omega,
n\in\mathbb N\right\}$, где
$q_{\lambda,n}(t):=\frac{1}{(t-\lambda)^n}$, $\lambda\in \mathbb
C$, $n\in\mathbb N$. При этом для множества $Q$ в линейном
пространстве $L$ символ ${\rm span}\,Q$ обозначает линейную
оболочку $Q$ в $L$.

Ниже будем использовать следующую терминологию из \cite[гл.\,3, \S\,5]{BEG}.
Кратным многообразием в
$\Omega$ называется конечная или бесконечная последовательность $W$
пар $(\lambda_k, m_k)$, где $Z(W):=\{\lambda_k\}$ --- дискретное подмножество $\Omega$ и
$m_k\in\mathbb N$ для любого $k$. Для непустого кратного многообразия $W=\{(\lambda_k, m_k)\}$ в $\Omega$
введем множество
$$
S(W):=\{f\in H(\Omega)\,|\, f^{(j)}(\lambda_k)=0,\,
0\le j\le m_k - 1 \,\mbox{ для любого }\, k\};
$$
$S(W)$ --- собственное замкнутое подпространство $H(\Omega)$.
Множества $S(W)$ --- это в точности все собственные замкнутые идеалы в топологической алгебре
$H(\Omega)$ с операцией обычного умножения функций;
они играют важную роль и при описании $D_{0,g_0}$-инвариантных подпространств $H(\Omega)$,
если $g_0$ имеет нули в $\Omega$.

Для удобства будем записывать кратное многообразие в $\Omega$ и так:
$W=\{(\lambda, m_\lambda)\,|\,\lambda\in\Lambda\}$, где $\Lambda$ --- конечное или счетное
дискретное подмножество $\Omega$, $m_\lambda\in\mathbb N$.

Для функции $h\in H(\Omega)$ символ $Z(h)$ обозначает множество всех нулей
$h$ в $\Omega$, $m(\lambda,h)$ --- кратность нуля $\lambda\in Z(h)$, а $W(h)$  ---
нулевое многообразие $h$, т.\,е. множество всех пар $(\lambda, m(\lambda, h))$, $\lambda\in Z(h)$.
Если $\lambda\in\mathbb C$ не является нулем $h$, то полагаем $m(\lambda, h):=0$.
Для множества $S\subset H(\Omega)$, $S\ne\{0\}$, через $Z(S)$ обозначим множество
всех общих нулей функций из $S$: $Z(S):=\bigcap\limits_{h\in S} Z(h)$. В случае $Z(S)\ne\emptyset$
кратным многообразием $S$
назовем множество $W(S)$ пар $(\lambda, m(\lambda, S))$, $\lambda\in Z(S)$, и положим
$m(\lambda, S):=\min\{m(\lambda, h)\,|\, h\in S\}$. Если $Z(S)=\emptyset$, то
считаем, что $W(S)=\emptyset$. Ясно, что всегда $S\subset S(W(S))$ (считаем, что
$S(\emptyset):=H(\Omega)$).
Для $\lambda\in\mathbb C\backslash Z(S)$ полагаем $m(\lambda, S):=0$.

\begin{remark} \label{RAS1} {\rm
Отметим следующее важное свойство "расщепляемости"\,
$D_{0,g_0}$: если $g_0=hv$, $f=hu$, $h, v, u\in H(\Omega)$, $v(0)=1$, то
$D_{0,g_0}^n(f)=h D_{0,v}^n(u)$ для любого $n\in\mathbb N$
(см. \cite[лемма 2]{ITOGI17}). Отсюда следует, что
для любого многочлена $P(z)=\sum\limits_{j=0}^N a_j z^j$ выполняется равенство
$P(D_{0, g_0})(f)=h P(D_{0, v})(u)$. Здесь, как обычно,
$P(D_{0, g_0})=\sum\limits_{j=0}^N a_j D_{0,g_0}^j$.
}
\end{remark}

\subsection{Экстремальные функции}

Всюду далее $g_0$ --- некоторая голоморфная в $\Omega$ функция, для которой $g_0(0)=1$.
Нам понадобятся функции из подпространств $H(\Omega)$, которые на заданном подмножестве
$Z(g_0)$ обращаются в $0$ c нужной кратностью, а на остальной части $Z(g_0)$
не равны $0$.

\begin{lemma}\label{extrem}
Пусть $G$ --- пространство Фреше, непрерывно вложенное в $H(\Omega)$;
$(\lambda_j)_j$ и $(\mu_k)_k$ --- дискретные последовательности попарно различных точек
$\Omega$ такие, что
$\lambda_j\ne\mu_k$ для любых $j, k$; $m_k\in\mathbb N$.
Предположим, что для любого $j$ существует функция $f_j\in G$ такая, что
$f_j(\lambda_j)\ne 0$, а для любого $k$ найдется функция $h_k\in G$, для которой $\mu_k$ -- нуль
кратности $m_k$, причем $\mu_k$ --- нуль кратности не меньше $m_k$ для любой функции $f\in G$.
Тогда существует функция $w\in G$, для которой $w(\lambda_j)\ne 0$ для любого $j$
и $\mu_k$ --- ее нуль кратности $m_k$ для каждого $k$.
\end{lemma}

\begin{proof} Пусть $\{p_n\,|\, n\in\mathbb N\}$ --- фундаментальная последовательность непрерывных
преднорм в $G$ (т.\,е. множества $\{x\in G\,|\, p_n(x)<\varepsilon\}$,
$\varepsilon > 0$, $n\in\mathbb N$, образуют базис окрестностей начала в $G$);
$p_n\le p_{n+1}$, $n\in\mathbb N$.
Положим $c_1:=\frac{1}{2(p_1(f_1)+1)}$ и $A_1:=c_1f_1(\lambda_1)$.
Если $c_j, A_j\in\mathbb C$ для $1\le j\le n$ для некоторого $n\in\mathbb N$
уже определены,
то определим $c_{n+1}\ne 0$, для которого
\begin{equation}
|c_{n+1}|\le \frac{1}{2^{n+1}(p_{n+1}(f_{n+1})+1)},
\end{equation}
\begin{equation}
|c_{n+1}|<\frac{1}{2^{n+1}}{\rm min}\left\{\frac{|A_j|}{|f_{n+1}(\lambda_j)|+1}\,|\,1\le j\le n \right\},
\end{equation}
и
\begin{equation}
A_{n+1}:=\sum\limits_{s=1}^{n+1}c_sf_s(\lambda_{n+1})\ne 0.
\end{equation}
Из (1.1) следует, что ряд $\sum\limits_{j} c_j f_j$ (возможно, конечный)
абсолютно сходится в $G$ к некоторой функции $f$. Из условий (1.2) и (1.3) вытекает, что $f(\lambda_j)\ne 0$
для любого $j$.


Аналогично построим функцию $h\in S$, исчезающую в точках $\mu_k$ с кратностями $m_k$.
Пусть $h_k(z)=(z-\mu_j)^{m_j}u_{k,j}(z)$, $u_{k,j}\in H(\Omega)$,
$u_{k,k}(\mu_k)\ne 0$. Полагаем $d_1:=\frac{1}{2(p_1(h_1)+1)}$,
$B_1:=d_1u_{1,1}(\mu_1)$.
Если числа $d_j$, $B_j$, $1\le j\le n$, для некоторого $n$ уже выбраны, то выбираем
$d_{n+1}$ такое, что
$$
|d_{n+1}|\le\frac{1}{2^{n+1}(p_{n+1}(h_{n+1})+1)};
$$
$$
|d_{n+1}|<\frac{1}{2^{n+1}}{\rm min}\left\{\frac{|B_s|}{|u_{n+1,s}(\mu_s)|+1}\,\left|\right.\,1\le s\le n\right\};
$$
$$
B_{n+1}:= \sum\limits_{s=1}^{n+1} d_s u_{s,n+1}(\mu_{n+1})\ne 0.
$$
Ряд $\sum\limits_k d_k h_k$ (возможно, конечный) сходится абсолютно в $G$ к некоторой функции $h\in G$.
При этом $\mu_k$ для любого $k$ --- нуль функции $h$ кратности $m_k$.

Пусть $f(z)=(z-\mu_k)^{m_k}v_k(z)$, $h(z)=(z-\mu_k)^{m_k}w_k(z)$,
$v_k, w_k\in H(\Omega)$, $w_k(\mu_k)\ne 0$.
Требуемая функция $w$ может быть найдена в виде $w=f+\beta h$, $\beta\in\mathbb C$.
При этом нужно выбрать $\beta\in\mathbb C$ такое, что для любых $j, k$
$$
f(\lambda_j)+ \beta h(\lambda_j)\ne 0;
$$
$$
v_k(\mu_k)+ \beta w_k(\mu_k)\ne 0,
$$
т.\,е. $\beta$  нужно выбирать вне не более чем счетного множества.
\end{proof}

\begin{definition} Пусть $S$ --- собственное замкнутое
подпространство $H(\Omega)$; одно из множеств $Z(S)$ и $Z(g_0)$ непусто.
Всякую функцию $v\in S$ такую,
что любое $\lambda\in Z(S)$ --- нуль кратности $m(\lambda, S)$
функции $v$ и $v$ в точках $Z(g_0)\backslash Z(S)$ в нуль не обращается, будем называть
$g_0$-экстремальной функцией (для) $S$.
\end{definition}

Лемма \ref{extrem} показывает, что всякое собственное замкнутое подпространство $H(\Omega)$,
если одно из множеств $Z(S)$ и $Z(W)$ непусто, содержит $g_0$-экстремальную функцию.

\begin{lemma}\label{LEMMAEXT} Пусть одно из множеств $Z(S)$ и $Z(g_0)$ непусто,
$v$ --- $g_0$-экстремальная функция для $S$. Тогда для любого
$f\in S$ существует $\alpha\in\mathbb C\backslash\{0\}$ такое,
что $f+\alpha v$ --- тоже $g_0$-экстремальная функция для $S$.
\end{lemma}

\begin{proof}
Пусть $W(S)=\{(\lambda_k, m_k) \}$. Для любого
$k$ функция $v$ представляется в виде
$v(z)=(z-\lambda_k)^{m_k}v_k(z)$, где $v_k\in H(\Omega)$,
$v_k(\lambda_k)\ne 0$, а функция $f$ имеет вид $f(z)=(z-\lambda_k)^{m_k}f_k(z)$,
$f_k\in H(\Omega)$. Так как $Z(g_0)$ не более чем счетно, то найдется
$\alpha\in\mathbb C\backslash\{0\}$, для которого
$f(\lambda)+\alpha v(\lambda)\ne 0$ для любого $\lambda\in Z(g_0)\backslash Z(S)$
и $f_k(\lambda_k) + \alpha v_k(\lambda_k)\ne 0$ для любого $k$.
(Если $W(S)$ пусто, то последние ограничения на $\alpha$ отсутствуют.)
\end{proof}

\medskip
Далее для кратных многообразий $W=\{(\lambda, n_\lambda)\}$ и $V=\{(\nu, m_\nu)\}$ в $\Omega$
будем писать $W\prec V$, если $Z(W)\subset Z(V)$ и
$n_\lambda\le m_\lambda$ для любого $\lambda\in Z(W)$.

Ниже символ $\mathcal D(g_0)$ обозначает множество всех многочленов $p$
таких, что $p(0)=1$, функция $g_0/p$ голоморфна в $\Omega$ и $p$ не имеет корней в $\mathbb C\backslash\Omega$.
Если $g_0\equiv 1$, то $\mathcal D(g_0)=\{g_0\}$.

\medskip
\begin{lemma} \label{MNOGOOBRASIE}
Предположим, что функция $g_0$ имеет нули в $\Omega$.
\begin{itemize}
\item[(i)] Если $S$ --- собственное замкнутое
$D_{0,g_0}$-инвариантное подпространство $H(\Omega)$ такое, что $S\backslash (g_0\mathbb C(z))\ne\emptyset$,
то множество $Z(S)$ непусто.

\item[(ii)] Пусть $S$ --- собственное замкнутое
$D_{0,g_0}$-инвариантное подпространство $H(\Omega)$ и множество $Z(S)$ непусто.
Тогда $W(S)\prec W(g_0)$.

\item[(iii)] Если $Z(g_0)$ бесконечно, то для любого собственного
замкнутого $D_{0,g_0}$-инвариан\-тного подпространства $S$
пространства $H(\Omega)$ множество $Z(S)$ непусто.

\item[(iv)] Пусть $S$ --- собственное замкнутое
$D_{0,g_0}$-инвариантное подпространство $H(\Omega)$, содержащееся в $g_0\mathbb C(z)$.
Тогда множество всех $\lambda\in Z(g_0)$ таких, что $m(\lambda, S)<m(\lambda, g_0)$,
конечно или пусто.
\end{itemize}
\end{lemma}

\begin{proof}
(i): Предположим, что $Z(S)=\emptyset$.
Тогда для любого $\lambda\in Z(g_0)$ найдется функция $f_\lambda\in S$ такая, что
$f_\lambda(\lambda)\ne 0$. По лемме \ref{extrem} существует $g_0$-экстремальная функция $w\in S$
для $S$, для которой
$w(\lambda)\ne 0$ для любого $\lambda\in Z(g_0)$.
Так как $w$ не является циклическим вектором $D_{0, g_0}$ в $H(\Omega)$, то
$w\in g_0\mathbb C(z)$ по теореме \ref{cycl}.
Зафиксируем $f\in S\backslash(g_0\mathbb C(z))$.
По лемме \ref{LEMMAEXT} найдется ненулевое $\alpha\in\mathbb C$
такое, что $w_0:= f + \alpha w\in S$ --- тоже $g_0$-экстремальная функция $S$,
т.\,е. $w_0(\lambda)\ne 0$ для любого $\lambda\in Z(g_0)$.
Кроме того, $w_0\notin g_0\mathbb C(z)$. Значит, по теореме \ref{cycl},
$w_0$ --- циклический вектор $D_{0,g_0}$ в $H(\Omega)$. Получили
противоречие.

(ii): Предположим, что существует $\lambda\in Z(S)\backslash Z(g_0)$.
Пусть $\lambda\ne 0$. Из равенства
$0=D_{0,g_0}(f)(\lambda)=\frac{f(\lambda)-g_0(\lambda)f(0)}{\lambda}=-\frac{f(0) g_0(\lambda)}{\lambda}$
следует, что $f(0)=0$ для любой функции $f\in S$.
Поэтому $D_{0, g_0}^n(f)(0)=0$ для любых функции $f\in S$
и целого $n\ge 0$.
Согласно \cite[лемма 2, доказательство леммы 3]{LINCHUK} (см. также \cite[лемма 7, замечание 10]{AA})
последовательность функционалов $\varphi_n: f\mapsto D_{0, g_0}^n(f)(0)$, $n\ge 0$,
полна в сопряженном $H(\Omega)'$, наделенном слабой топологией $\sigma(H(\Omega)', H(\Omega))$.
Значит, всякая функция $f\in S$ является тождественным нулем. Противоречие.
Пусть $\lambda=0$. Тогда $f(t)/t^n\in S$, т.\,е. $f^{(n)}(0)=0$ для всех $f\in S$ и
целых $n\ge 0$. Поэтому $S=\{0\}$ и снова получаем противоречие.

Возьмем $\lambda\in Z(S)$. Предположим, что $m(\lambda, S)>m(\lambda, g_0)$.
Так как $\lambda\ne 0$, то
$D_{0,g_0}(f)(\lambda)=\frac{f(\lambda)-g_0(\lambda)f(0)}{\lambda}$. Поскольку
$\lambda$ --- нуль кратности не меньше $m(\lambda, S)$ функции $D_{0,g_0}(f)$ для всех $f\in S$, то
$f(0)=0$, если $f\in S$. Поэтому $D_{0, g_0}^n(f)(0)=0$ для любой функции $f\in S$
и любого целого $n\ge 0$. Противоречие.

(iii): Пусть $Z(g_0)=\{\lambda_k\,|\,k\in\mathbb N\}$. Предположим, что $Z(S)$ пусто.
Тогда для любого $k\in\mathbb N$ найдется функция $f_k\in S$ такая, что
$f_k(\lambda_k)\ne 0$. По лемме \ref{extrem}  существует $g_0$-экстремальная функция $w\in S$
для $S$, для которой
$w(\lambda_k)\ne 0$ для любого $k\in\mathbb N$. Отсюда следует, что $w\notin \mathcal Z_0\cup(g_0\mathbb C[z])$.
Теорема \ref{cycl} влечет, что $w$ --- циклический вектор $D_{0,g_0}$
в $H(\Omega)$. Противоречие.

(iv): Предположим противное. Тогда все нули $\lambda$ функции $g_0$, для которых
$m(\lambda, S)<m(\lambda, g_0)$, образуют последовательность
$\{\lambda_j\,|\, j\in\mathbb N\}$ различных точек.
Положим
$k_j:= m(\lambda_j, g_0)-m(\lambda_j, S)$, $j\in\mathbb N$. Существует голоморфная в $\Omega$
функция $u$ с нулевым многообразием $\{(\lambda_j, k_j)\,|\, j\in\mathbb N\}$
и такая, что $u(0)=1$. Пусть $v$ --- $g_0$-экстремальная функция $S$.
Тогда функция $h_0:=\frac{u v}{g_0}$ голоморфна в $\Omega$. При этом
$D_{0, g_0}^n(v)=\frac{g_0}{u} D_{0,u}^n({h_0})$ для любого целого $n\ge 0$.
Функции $u$ и $h_0$ не имеют общих нулей. Кроме того, не существует рациональной функции
$R$ такой, что $h_0=R u$, т.\,е. $\frac{v}{g_0}=R$.
Действительно, каждая точка $\lambda_j$, $j\in\mathbb N$,
является полюсом $\frac{v}{g_0}$ в $\Omega$. Из теоремы \ref{cycl} следует, что
$ \frac{g_0}{u} H(\Omega)\subset S$.
Получено противоречие с вложением $S\subset g_0\mathbb C(z)$.
\end{proof}

Из леммы \ref{MNOGOOBRASIE}\,(iv) вытекает

\begin{corollary} \label{SLEDMNOGOOBRASIE}
Для любого собственного замкнутого $D_{0,g_0}$-инвариантного подпространства $S$
пространства $H(\Omega)$, содержащегося в $g_0\mathbb C(z)$,
существует единственный многочлен $p_S\in\mathcal D(g_0)$ такой, что
$W(S)=W(g_0/p_S)$.
\end{corollary}

\begin{remark} {\rm
Если множество нулей функции $g_0$ конечно, то существуют
собственные замкнутые $D_{0,g_0}$-инвариантные подпространства $H(\Omega)$
с пустым множеством общих нулей. Пусть $g_0=q g_1$, где
$q$ --- многочлен степени не меньше $1$, $q(0)=1$, функция $g_1\in H(\Omega)$ не
имеет нулей в $\Omega$ и $g_1(0)=1$. Как будет показано далее (см. теоремы \ref{WITHZERO},
\ref{PLANE}),
для любого $n\ge {\rm deg}(q) - 1$ множество
$S= g_1\mathbb C[z]_n$ является собственным замкнутым
$D_{0,g_0}$-инвариантным подпространством $H(\Omega)$. При этом $Z(S)=\emptyset$.
}
\end{remark}

\medskip
\subsection{Оператор, сопряженный к $D_0$}

Будем писать $D_0$ вместо $D_{0,g_0}$, если функция $g_0$ тождественно равна $1$.
Заметим, что $D_{0,g_0}= D_0-A_0$, где $A_0(f)=f(0)g_1=\delta_0(f)g_1$
и $g_1(t)=\frac{g_0(t)-1}{t}$, если $t\ne 0$. При этом $\delta_0(f)=f(0)$.
Именно в таком виде оператор
$D_{0,g_0}$ был исследован Ю.С. Линчуком \cite{LINCHUK}.
Если рассматривать естественную двойственность между $H(\Omega)$ и
$H(\Omega)'$, сопряженный оператор $D_{0,g_0}'$  имеет следующий вид:
$D_{0,g_0}'=D_0'- A_0'$, где $A_0'(\varphi)=\varphi(g_1)\delta_0$,
$\varphi\in H(\Omega)'$.

Пусть $H_0:=H(\{0\})$ --- пространство ростков всех функций, голоморфных в начале.
Оно  наделяется естественной топологией индуктивного предела (см., например,
\cite[\S~1]{KHAVIN}).

По теореме Силва-Кете-Гротендика о двойственности пространств аналитических функций
\cite{KOETHE} (см. также \cite[\S\,2]{KHAVIN}) преобразование Коши
$$
\varphi\mapsto \varphi(q_{\lambda,1}), \, \lambda\ne 0, \, \varphi\in H_0',
$$
является топологическим изоморфизмом сильного сопряженного к $H_0$ на
пространство Фреше $H_0(\overline{\mathbb C}\backslash\{0\})$ всех
функций, голоморфных в $\overline{\mathbb C}\backslash\{0\}$ и равных $0$ в $\infty$.
Возникающая при таком изоморфизме билинейная форма
$$
\langle f, h\rangle:=-\frac{1}{2\pi i}\int\limits_{|t|=\varepsilon}f(t)h(t)dt,
\, f\in H_0,\, h\in H_0(\overline{\mathbb C}\backslash\{0\}),
$$
задает двойственность между $H_0$ и $H_0(\overline{\mathbb C}\backslash\{0\})$
($\varepsilon>0$ выбирается так, чтобы $f$ была голоморфна в некоторой области, содержащей круг
$|t|\le\varepsilon$, а окружность $|t|=\varepsilon$ обходится против часовой стрелки).
Отметим, что
\begin{equation} \label{DUALITY}
\langle q_{\lambda,k}, h\rangle =\frac{h^{(k-1)}(\lambda)}{(k-1)!},\,\, \lambda\ne 0,\, k\in\mathbb N, \,
h\in H_0(\overline{\mathbb C}\backslash\{0\}).
\end{equation}

Для линейного непрерывного оператора $A: H_0\to H_0$ символом $A'$ обозначим сопряженный к $A$
(относительно дуальной системы $(H_0, H_0(\mathbb C\backslash\{0\}))$)
оператор из $H_0(\mathbb C\backslash\{0\})$ в $H_0(\mathbb C\backslash\{0\})$.

Из равенства $\langle D_0(f), h\rangle=
\langle f, D_0'(h)\rangle$, $f\in H_0$, $h\in H_0(\overline{\mathbb C}\backslash\{0\})$,
для функций $f(t)=\frac{1}{t-\lambda}$ следует, что
$D_0'(h)(\lambda)=\frac{1}{\lambda}h(\lambda)$ для $\lambda\ne 0$. Данная двойственность является удобным средством изучения
оператора $D_0$ (см. далее замечание \ref{drobi}).

\medskip

Обозначим через $\mathbb C[D_{0,g_0}]$ множество всех операторов
$P(D_{0,g_0})$, $P\in\mathbb C[z]$, символом $I$ --- тождественный оператор.
Для $\lambda\in\mathbb C\backslash\{0\}$, $k\in\mathbb N$ положим
$$
Q(\lambda, k):={\rm span}\left\{ q_{\lambda, j}\,|\, 1\le j\le k\right\}.
$$

\medskip
\begin{remark}\label{drobi}{\rm
Отметим простые факты о действии оператора $D_0$ в $H_0$, в частности, о его действии на простейшие дроби.

\noindent
1) ${\rm Ker}\,D_0^n=\mathbb C[z]_{n-1}$ для любого $n\in\mathbb N$.

\noindent
2) Если $\lambda\ne 0$, то $1/\lambda$ --- собственное значение $D_0: H_0\to H_0$, а
 собственными векторами оператора $D_0$, соответствующими $1/\lambda$,
 являются функции $\frac{C}{t-\lambda}$, $C\in\mathbb C\backslash\{0\}$,
и только они.


\medskip
Следующие свойства связаны с "просеиванием"\, рациональных дробей с помощью
операторов из $\mathbb C[D_0]$.

\medskip

\noindent 3) Для любых $\lambda, \mu\in\mathbb C\backslash\{0\}$,
$k\in\mathbb N$ существуют числа $\alpha_j$, $1\le j\le k$, такие,
что
$$
\left(D_0 -\frac{1}{\mu}I\right)(q_{\lambda, k})=\sum\limits_{j=1}^k\alpha_j q_{\lambda, j}.
$$
При этом $\alpha_k\ne 0$ в случае $\mu\ne\lambda$.

\noindent
4) $\left(D_0 -\frac{1}{\mu}I\right)(Q(\lambda, k))\subset Q(\lambda, k)$
для любых $\lambda, \mu\in\mathbb C\backslash\{0\}$, $k\in\mathbb N$.

Если при этом $\lambda\ne\mu$, $f\in Q(\lambda, k)$, $f\ne 0$, то
$\left(D_0 -\frac{1}{\mu}I\right)(f)\ne 0$.

\noindent 5) $\left(D_0-\frac{1}{\lambda}I\right)^n(q_{\lambda,
n})=0$ для любых $\lambda\ne 0$, $n\in\mathbb N$.

\noindent
6) Для любых $\lambda\ne 0$, $k, m\in\mathbb N$ таких, что $1\le m\le k$, любых $b_r\in\mathbb C$, $1\le r\le k$,
$b_k\ne 0$, для функции $R(t)=\sum\limits_{r=1}^k b_r q_{\lambda, r}$
существует оператор $A=\sum\limits_{j=0}^k a_j D_0^j\in\mathbb C[D_0]$ такой, что $A(R)=q_{\lambda,m}$.
}
\end{remark}

\begin{proof}
3): Для любого $h\in H_0(\overline{\mathbb C}\backslash\{0\})$, любого $\varepsilon\in(0,|\lambda|)$
$$
\left\langle \left(D_0-\frac{1}{\mu}I\right)\left(q_{\lambda, k}\right), h\right\rangle =
\left\langle q_{\lambda, k}(t), \left(\frac{1}{t} -\frac{1}{\mu}\right)h(t)\right\rangle=
$$
$$
\left\langle q_{\lambda,k}(t),\frac{h(t)}{t}\right\rangle - \frac{1}{\mu}\left\langle q_{\lambda,k}, h\right\rangle=
\frac{1}{(k-1)!}h_1^{(k-1)}(\lambda) - \frac{1}{\mu}\left\langle q_{\lambda,k}, h\right\rangle,
$$
где $h_1(t)=h(t)/t$.
Поскольку
$$
h_1^{(k-1)}(\lambda)=\sum\limits_{j=0}^{k-1} C_{k-1}^j h^{(j)}(\lambda)\frac{(-1)^{k-1-j}(k-1-j)!}{\lambda^{k-j}},
$$
то, с учетом (\ref{DUALITY}), получим:
$$
\left\langle \left(D_0-\frac{1}{\mu}I\right)\left(q_{\lambda, k}\right), h\right\rangle =
$$
$$
\frac{1}{(k-1)!}\sum\limits_{j=0}^{k-1} C_{k-1}^j\frac{(-1)^{k-1-j}(k-1-j)!}{\lambda^{k-j}}h^{(j)}(\lambda)-
\frac{1}{\mu}\left\langle q_{\lambda,k}, h\right\rangle=
$$
$$
\frac{1}{(k-1)!}\sum\limits_{j=0}^{k-1} C_{k-1}^j\frac{(-1)^{k-1-j}(k-1-j)! j!}{\lambda^{k-j}}\left\langle q_{\lambda,j-1}, h\right\rangle
-\frac{1}{\mu}\left\langle q_{\lambda,k}, h\right\rangle.
$$
Отсюда следует, что
$\left(D_0-\frac{1}{\mu}I\right)\left(q_{\lambda,k}\right)=
\sum\limits_{j=1}^k\alpha_j q_{\lambda,j}$, где
$\alpha_k=\frac{1}{\lambda}-\frac{1}{\mu}\ne 0$, если $\lambda\ne\mu$,
и
$\alpha_j=\frac{(-1)^{k-j}}{\lambda^{k-j+1}}$,
\,$1\le j\le k-1$ (для $k\ge 2$).

\medskip
Утверждение 4) следует из 3).

\medskip
Равенство в 5) следует из того, что для любого $h\in H_0(\overline{\mathbb C}\backslash\{0\})$
$$
\left\langle \left(D_0-\frac{1}{\lambda} I\right)^{n}(q_{\lambda,n}), h\right\rangle=
\left\langle q_{\lambda,n},\left(\left(D_0-\frac{1}{\lambda} I\right)'\right)^n (h)\right\rangle=
$$
$$
\left\langle q_{\lambda,n}(t),\left(\frac{1}{t}-\frac{1}{\lambda}\right)^n h(t)\right\rangle=
\frac{(-1)^{n+1}}{2\pi i \lambda^n}\int\limits_{|t|=\varepsilon}\frac{h(t)}{t^n}dt=0.
$$

\noindent
6): Равенство $A(R)=q_{\lambda,m}$ равносильно тому, что
$\langle A(R), h\rangle=\frac{h^{(m-1)}(\lambda)}{(m-1)!}$ для каждой функции $h\in H_0(\overline{\mathbb C}\backslash\{0\})$,
т.\,е. тому, что для любой функции $h\in H_0(\overline{\mathbb C}\backslash\{0\})$
\begin{equation}
\sum\limits_{r=1}^k b_r\langle A(q_{\lambda, r}), h\rangle=\frac{h^{(m-1)}(\lambda)}{(m-1)!}.
\label{SVOISTVO8}
\end{equation}
Равенство (\ref{SVOISTVO8}) для функции $w(t)=\sum\limits_{j=0}^k\frac{a_j}{t^j}$
можно переписать так:
$$
\sum\limits_{r=1}^k \frac{b_r}{(r-1)!}(w h)^{(r-1)}(\lambda)= \frac{h^{(m-1)}(\lambda)}{(m-1)!},
$$
а значит, в виде
$$
\sum\limits_{s=0}^{k-1} h^{(s)}(\lambda)\sum\limits_{r=0}^{k-s-1}\frac{b_{r+s+1}}{(r+s)!}C_{r+s}^s
w^{(r)}(\lambda)=\frac{h^{(m-1)}(\lambda)}{(m-1)!}.
$$
Приравняем множитель (слева) при $h^{(m-1)}(\lambda)$ к
$1/(m-1)!$, при $h^{(s)}(\lambda)$, $0\le s\le~{k-1}$, $s\ne m-1$
--- к $0$. Получим систему линейных уравнений с $k$ неизвестными
$w^{(r)}(\lambda)$, $0\le r\le k-1$, с верхнетреугольной матрицей.
Диагональные элементы $\frac{b_k}{(k-1)!}C_{k-1}^s$, $0\le s\le
k-1$, этой матрицы отличны от $0$. Следовательно, эта система
имеет ненулевое решение, по которому определим коэффициенты $a_j$,
$0\le j\le k$.
\end{proof}

\subsection{Вспомогательные результаты для инвариантных подпрост\-ранств,
содержащихся в $g_0\mathbb C(z)$}

\medskip
\begin{lemma} \label{STEPEN}
Пусть $\Omega\ne\mathbb C$,
$f=\frac{g_0}{p}r + g_0 h$, где $p\in\mathcal D(g_0)$, $r$ --- ненулевой многочлен и
$h\in\mathbb C^-_\Omega(z)$. Тогда
существует оператор $A\in\mathbb C[D_{0,g_0}]$ такой, что $A(f)=\frac{g_0}{p}\widetilde r$,
где $\widetilde r$ --- многочлен степени $\max\{{\rm deg}(r); {\rm deg}(p)-1\}$.
\end{lemma}

\begin{proof}
Разложим $h$ на простейшие дроби: $h=\sum\limits_{j=1}^m\sum\limits_{k=1}^{k_j} a_{j,k} q_{\lambda_j, k}$,
где $\lambda_j$ --- различные точки в $\mathbb C\backslash\Omega$, $k_j\in\mathbb N$, $1\le j\le m$.
Вследствие замечания \ref{drobi},\,5) для оператора
$$
B:=\left(D_{0, g_0}-\frac{1}{\lambda_1} I\right)^{k_1}\left(D_{0, g_0}-\frac{1}{\lambda_2} I\right)^{k_2}\cdot\cdot\cdot
\left(D_{0, g_0}-\frac{1}{\lambda_m} I\right)^{k_m}
$$
выполняется равенство $B(g_0 h)=0$. Поэтому $B(f)= B\left(\frac{g_0}{p}r\right)$.

Отметим, что для $\mu, t\in\mathbb C\backslash\{0\}$
$$
\left(D_{0, g_0}-\mu I\right)\left(\frac{g_0}{p}r\right)(t)=
$$
$$
\frac{g_0(t)}{p(t)}
\left(D_{0, p}-\mu I\right)(r)(t)=
\frac{g_0(t)}{p(t)}\left(\frac{r(t)-p(t)r(0)}{t} -\mu r(t)\right)=
$$
\begin{equation}
\frac{g_0(t)}{p(t)}\frac{r(t)-\mu t r(t)- p(t)r(0)}{t}=\frac{g_0(t)}{p(t)}r_1(t),
\label{17.04}
\end{equation}
где $r_1(t):=\frac{r(t)-\mu t r(t)- p(t)r(0)}{t}$.
Положим $n:={\rm deg}(r)$. Если $n\ge {\rm deg}(p)$, то ${\rm deg}(r_1)=n$.
Поэтому $B(f)=\frac{g_0}{p}\widetilde r$, где $\widetilde r$ --- многочлен степени $n$.
В этом случае $A=B$.

Пусть теперь $n\le {\rm deg}(p) - 1$. Вследствие (\ref{17.04}) $B(f)=\frac{g_0}{p}r_2$,
где $r_2$  --- многочлен степени не большей, чем ${\rm deg}(p) - 1$.
Если ${\rm deg}(r_2)={\rm deg}(p) - 1$, то лемма доказана (и $A=B$).

Пусть ${\rm deg}(r_2)< {\rm deg}(p) - 1$.
Вначале покажем, что многочлен $r_2$  ненулевой. Предположим противное. Тогда
существуют $j\in\{1,...,m\}$ и ненулевой многочлен $q$ такие, что
функция $\left(D_{0,g_0} - \frac{1}{\lambda_j}I \right)\left(\frac{g_0}{p} q\right)$
является тождественным нулем. Вследствие равенства (\ref{17.04}) (для $q$ и $1/\lambda_j$
вместо $r$ и $\mu$ соответственно) для любого $t\in\Omega$
$$
\left(1-\frac{t}{\lambda_j}\right) q(t)= q(0) p(t).
$$
Поэтому $q(0)\ne 0$, а значит, $p(\lambda_j)=0$. Получено противоречие с тем, что
$p$ не имеет нулей в $\mathbb C\backslash\Omega$.

Подействуем на $B(f)$ оператором $D_{0,g_0}$:
$$
D_{0,g_0}(B(f))(t)=\frac{g_0(t)}{p(t)}\frac{r_2(t)-p(t) r_2(0)}{t}.
$$
Положим $r_3(t):=\frac{r_2(t)-p(t) r_2(0)}{t}$.
Если $r_2(0)\ne 0$, то ${\rm deg}(r_3)={\rm deg}(p) - 1$. В этом случае
$A= D_{0,g_0} B$, $\widetilde r=r_3$.
Пусть $r_2(t)=t^s r_0(t)$, где $s\in\mathbb N$ и $r_0$ ---
многочлен такой, что $r_0(0)\ne 0$. Тогда
$D_{0,g_0}^s(B(f))=\frac{g_0}{p}r_0$, $t\ne 0$, и
$$
D_{0,g_0}^{s+1}(B(f))(t)=\frac{g_0(t)}{p(t)}\frac{r_0(t) - p(t)r_0(0)}{t}.
$$
Степень многочлена $\widetilde r(t)=\frac{r_0(t) - p(t)r_0(0)}{t}$
равна ${\rm deg}(p) - 1$. В этой ситуации $A=D_{0, g_0}^{s+1} B$.
\end{proof}

\medskip
\begin{lemma} \label{STRUCTURA}
Пусть $S$ --- собственное замкнутое $D_{0,g_0}$-инвариантное подпространство
$H(\Omega)$, содержащееся в $g_0\mathbb C(z)$; $p_S\in\mathcal D(g_0)$ --- многочлен такой,
что $W(S)=W(g_0/p_S)$ (см. следствие \ref{SLEDMNOGOOBRASIE}).

\begin{itemize}
\item[(i)] Для любой ненулевой функции $f\in S$ найдутся многочлены $r=r(f)$,
$u=u(f)$, $v=v(f)$ такие, что ${\rm deg}(u)<{\rm deg}(v)$, $v$ унитарный (т.\,е.
старший коэффициент $v$ равен 1),
$v$ не имеет корней в $\Omega$, многочлены $u$ и $v$ не имеют общих корней и
\begin{equation} \label{PREDSTAVLENIE}
f=\frac{g_0}{p_S}r + g_0\frac{u}{v}.
\end{equation}
\item[(ii)] В представлении (\ref{PREDSTAVLENIE}) многочлены $r$, $u$, $v$
определены однозначно.
\end{itemize}

\end{lemma}

\begin{proof} (i):
Пусть $p:=p_S$.
Возьмем $f\in S$. Тогда для некоторых многочленов
$r_1, u_1, v_1$ таких, что ${\rm deg}(u_1)<{\rm deg}(v_1)$ и $u_1$ и $v_1$ не имеют
общих корней, выполняется равенство
$f=g_0\left(r_1+\frac{u_1}{v_1}\right)$.
Факторизуем $v_1$ в виде $v_1=q v$, где  $q\in\mathcal D(g_0)$,
а унитарный многочлен $v$ не имеет корней в $\Omega$. Тогда
\begin{equation}
\label{raslozh}
f=g_0\left(r_1+\frac{u_1}{q v}\right)=\frac{g_0}{p}p r_1 + g_0\frac{u_1}{q v}.
\end{equation}
Правильную дробь $\frac{u_1}{qv}$ представим в виде $\frac{u_1}{qv}=\frac{q_1}{q}+\frac{u}{v}$,
где $q_1$, $u$ --- многочлены, для которых ${\rm deg}(q_1)<{\rm deg}(q)$
и ${\rm deg}(u)<{\rm deg}(v)$. Получим, что
$$
f=\frac{g_0}{p}p r_1 + g_0\left(\frac{q_1}{q}+\frac{u}{v}\right)=
\frac{g_0}{p}\left(p r_1 +\frac{p q_1}{q}\right) +g_0\frac{u}{v}.
$$
При этом многочлен $p$ делится на $q$.
Действительно, функция $\frac{p f}{g_0}$
голоморфна в $\Omega$. Из (\ref{raslozh}) следует, что в $\Omega$
голоморфна функция $pr_1 + \frac{p u_1}{q v}$, а значит, и $\frac{ p u_1}{q v}$.
Поскольку $u_1$ и $q$ не имеют общих корней, то $p$ делится на $q$.
Поэтому $r:=p r_1 +\frac{p q_1}{q}$ ---
многочлен.

(ii): Докажем единственность представления (\ref{PREDSTAVLENIE}).
Пусть для многочленов $r_1$, $u_1$, $v_1$ и $r_2$, $u_2$, $v_2$, как в (i),
в $\Omega$ выполняется равенство
$$
\frac{g_0}{p}r_1 + g_0\frac{u_1}{v_1}=\frac{g_0}{p}r_2 + g_0\frac{u_2}{v_2}.
$$
Тогда в $\Omega$
$$
\frac{r_1 v_1+p u_1}{v_1}=\frac{r_2 v_2+p u_2}{v_2},
$$
причем числитель и знаменатель в обеих дробях не имеют общих корней.
Это влечет, что $v_1=v_2=:\widetilde v$.
Поэтому $(r_1-r_2)\widetilde v=p(u_2-u_1)$. Так как многочлены $p$ и $\widetilde v$
взаимно простые, то $u_2-u_1$ делится на $\widetilde v$, а значит, $u_1=u_2$. Тогда
и $r_1=r_2$.
\end{proof}

\medskip
Пусть $S$ --- собственное замкнутое $D_{0,g_0}$-инвариантное подпространство
$H(\Omega)$, содержащееся в $g_0\mathbb C(z)$.
Символом $\mathcal P(S)$ обозначим множество всех корней многочленов $v(f)$, $f\in S\backslash\{0\}$,
как в (\ref{PREDSTAVLENIE}),
т.\,е.
$$
\mathcal P(S):=\{\lambda\in\mathbb C\backslash\Omega\,|\,\exists f\in S\backslash\{0\}:\, v(f)(\lambda)=0\}.
$$
Далее ${\bf 1}$ обозначает функцию, тождественно равную $1$. Считаем, что $r(f)=u(f)=0$ и $v(f)={\bf 1}$,
если $f=0$.

Пусть $\Upsilon$ --- конечное кратное многообразие в $\mathbb C\backslash\Omega$, т.\,е.
$\Upsilon=\{(\lambda, n_\lambda)\,|\,\lambda\in\Lambda\}$, где $n_\lambda\in\mathbb N$,
$\Lambda$ --- конечное подмножество $\mathbb C\backslash\Omega$.
Положим
$$
\mathbb C^-_{\Upsilon}(z):={\rm span}\left\{ q_{\lambda,k}\,|\,\lambda\in\Lambda,\, 1\le k\le n_\lambda\right\}.
$$
Если $\Upsilon$ пусто, т.\,е. $\Lambda$ --- пустое множество, то для удобства
считаем, что $\mathbb C_\Upsilon^-(z)=\{0\}$.

Полагаем $\mathbb C[z]_{-\infty}:=\{0\}$.

\medskip
\begin{remark} {\rm
Ниже (теоремы 2--4) будет показано, что семейство пространств
$$
S(p, n, \Upsilon):=\frac{g_0}{p}\mathbb C[z]_n+g_0\mathbb
C_\Upsilon^-(z),
$$
где $n\in\mathbb N\bigcup\{-\infty, 0\}$,
$p\in\mathcal D(g_0)$, $\Upsilon$ --- конечное или пустое кратное
многообразие в $\mathbb C\backslash\Omega$, содержит все
собственные замкнутые $D_{0,g_0}$-инвариантные подпространства
$H(\Omega)$, вложенные в $g_0\mathbb C(z)$.

Отметим, что $S(n, p, \Upsilon)$ однозначно определяется тройкой $(n,p,\Upsilon)$,
т.\,е. пространства $S(p_1, n_1, \Upsilon_1)$ и $S(p_2, n_2, \Upsilon_2)$
совпадают тогда и только тогда, когда $n_1=n_2$, $p_1=p_2$, $\Upsilon_1=\Upsilon_2$.
}
\end{remark}

\medskip
\begin{lemma}\label{POLDROBI}
Пусть собственное замкнутое $D_{0,g_0}$-инвариантное подпространство $S$
пространства $H(\Omega)$ содержится в $g_0\mathbb C(z)$; $p_S\in\mathcal D(g_0)$ ---
многочлен, для которого $W(S)=W(g_0/p_S)$; для $f\in S\backslash\{0\}$ многочлены
$r(f)$, $u(f)$, $v(f)$ такие, как в равенстве (\ref{PREDSTAVLENIE}).
Тогда
\begin{itemize}
\item[(i)] $n(S):=\sup\limits_{f\in S}{\rm deg} (r(f))<+\infty$ и $\frac{g_0}{p_S}\mathbb C[z]_{n(S)}\subset S$.

При этом $n(S)\ge {\rm deg}(p)-1$, если $n(S)\ne-\infty$.
\item[(ii)] $\mathcal P(S)$ конечно или пусто.
\item[(iii)] $n_\lambda(S):=\sup\limits_{f\in S} m(\lambda, v(f))<+\infty$
для любого $\lambda\in\mathcal P(S)$.
\end{itemize}
\end{lemma}

\begin{proof}
(i): Положим $p:=p_S$.
Предположим, что $S$ содержит функцию $f$,
для которой $r:=r(f)$ --- ненулевой
многочлен.
В силу леммы \ref{STEPEN} подпространство $S$ содержит и функцию $\widetilde f=\frac{g_0}{p}\widetilde r$,
где многочлен $\widetilde r$ имеет степень $m\ge {\rm deg}(p)-1$.
Поскольку $W(g_0/p)\prec W(\widetilde f)$, то многочлен $\widetilde r$ не имеет общих корней с $p$,
а следовательно, взаимно прост с $p$.

Пусть $k:={\rm deg}(p)\ge 1$.
По \cite[лемма 6]{ITOGI17} система $\left\{D_{0,p}^j(\widetilde r)\,|\, 1\le j\le k\right\}$ линейно независима
в $H(\Omega)$.
Если $m= k-1$, то
$\left\{D_{0,p}^j(\widetilde r)\,|\, 1\le j\le k\right\}$ --- базис в $\mathbb C[z]_m$.
Если $m=k$, то $\{D_{0,p}^j(\widetilde r)\,|\, 0\le j\le k\}$ --- базис в $\mathbb C[z]_m$.

Пусть теперь $m>k$. В этом случае ${\rm deg}\left(D_{0,p}^j(\widetilde r)\right)=m-j$, $0\le j\le m-k$. Поскольку по
\cite[лемма 4]{ITOGI17} многочлены $D_{0,p}^{m-k}(\widetilde r)$ и $p$ взаимно простые, то система
$\{ D_{0, p}^j(\widetilde r)=D_{0,p}^{j-m+k}(D_{0,p}^{m-k}(\widetilde r))\,|\, m-k+1\le j\le m\}$,
которая содержится в $\mathbb C[z]_{k-1}$, линейно независима в $H(\Omega)$.
Отсюда следует, что $\{D_{0,p}^j(\widetilde r)\,|\, 0\le j\le m\}$ --- базис в $\mathbb C[z]_m$
и в случае $m>k$. Поскольку для любого $j\in\mathbb N$ выполняется равенство
$D_{0,g_0}^j(\widetilde f)=\frac{g_0}{p}D_{0,p}^j(\widetilde r)$, то
$\frac{g_0}{p}\mathbb C[z]_m\subset S$.
Если ${\rm deg}(p)=0$, т.\,е. $p\equiv 1$, то $D_{0,g_0}^j\left(\frac{g_0}{p}\widetilde r\right)=
g_0 D_0^j(\widetilde r)$, $j\in\mathbb N$. Поскольку ${\rm deg}(D_0^j(\widetilde r)=m-j$,
$0\le j\le m$, то также $\frac{g_0}{p}\mathbb C[z]_m\subset S$.

Предположим, что $\sup\limits_{f\in S}{\rm deg}(r(f))=+\infty$.  Тогда существуют функции
$f_n\in S$, $n\in\mathbb N$, такие, что $m_n:={\rm deg}(r(f_n))\to +\infty$. Без
ограничения общности $m_n\ge k-1$ для любого $n\in\mathbb N$.
По доказанному выше $\frac{g_0}{p}\mathbb C[z]_{m_n}\subset S$ для всех $n\in\mathbb N$, а
значит, $\frac{g_0}{p}\mathbb C[z]\subset S$. Так как область $\Omega$ односвязная, то по теореме Рунге
$\mathbb C[z]$ плотно в $H(\Omega)$. Поэтому $\frac{g_0}{p}H(\Omega)\subset S$
вследствие замкнутости $S$.
Это противоречит вложению $S\subset g_0\mathbb C(z)$. Итак, $n(S)=\sup\limits_{f\in S}{\rm deg}(r(f))<+\infty$.

Вложение $\frac{g_0}{p}\mathbb C[z]_{n(S)}\subset S$ в случае $n(S)=-\infty$ очевидно.
Если $n(S)\ge 0$, то, как доказано выше,
$n(S)\ge {\rm deg}(p)-1$ и $\frac{g_0}{p}\mathbb C[z]_{n(S)}\subset S$.

(ii): Зафиксируем $\lambda\in\mathcal P(S)$ и возьмем функцию $f\in S$, для которой
$v(f)(\lambda)=0$ (многочлены $r(f)$, $u(f)$, $v(f)$ такие, как в (\ref{PREDSTAVLENIE})).
Вследствие (i) $\frac{g_0}{p}r(f)\in S$, а значит,
$g_0\frac{u(f)}{v(f)}\in S$.
Пусть $\lambda, \lambda_1, ...,\lambda_s$ --- все различные корни $v(f)$.
Вследствие замечаний \ref{RAS1} и \ref{drobi},\, 4), 5) для некоторого $n\in\mathbb N$ функция
$f_1=\left(D_{0,g_0}-\frac{1}{\lambda_1}I\right)^m\cdot\cdot\cdot
\left(D_{0,g_0}-\frac{1}{\lambda_s}I\right)^m\left(g_0\frac{u(f)}{v(f)}\right)$
имеет вид $f_1=g_0 h$, где $h$ --- ненулевая дробь из $Q(\lambda, k)$.
Применяя замечание \ref{drobi},\,6),
найдем оператор $A=\mathbb C[D_{0,g_0}]$ такой, что
$A(R_2)=g_0q_{\lambda, 1}$. Значит,
все функции $q_{\lambda,1}$, $\lambda\in\mathcal P(S)$,
содержатся в $S$.

Предположим, что $\mathcal P(S)$ бесконечно. Тогда $\mathcal P(S)$ имеет предельную точку в
$\overline{\mathbb C}\backslash\Omega$, и следовательно, множество
$\left\{q_{\lambda,1}\,|\,\lambda\in\mathcal P(S)\right\}$ полно
в $H(\Omega)$.
Это влечет, что $g_0 H(\Omega)\subset S$. Противоречие.

(iii): Если $\mathcal P(S)$ пусто, то $v(f)={\bf 1}$ для всех $f\in S$.
Пусть $\mathcal P(S)$ непусто, а значит, вследствие (ii), конечно. Предположим, что
$\sup\limits_{f\in S}{\rm deg}(v(f))=+\infty$. Тогда существуют $\lambda\in\mathcal P(S)$,
неограниченная возрастающая последовательность чисел $k_n\in\mathbb N$, для которых
$g_0 q_{\lambda, k_n}\in S$ для любого $n\in\mathbb N$.
Замечание \ref{drobi}, 6) влечет тогда, что все функции
$g_0q_{\lambda, k}$, $k\in\mathbb N$, содержатся в $S$. Поскольку множество
$\{q_{\lambda, k}\,|\, k\in\mathbb N\}$ полно в $H(\Omega)$ и $S$ замкнуто, то
$g_0 H(\Omega)\subset S$. Противоречие.
\end{proof}

\medskip
Далее для собственного замкнутого $D_{0,g_0}$-инвариантного подпространства $S$
пространства $H(\Omega)$, содержащегося в $g_0\mathbb C(z)$, введем его конечное
или пустое сингулярное кратное многообразие
$$
\Upsilon(S):=\{(\lambda, n_\lambda(S))\,|\,\lambda\in\mathcal P(S)\},
$$
где $n_\lambda(S)$ такие, как в лемме \ref{POLDROBI}\,(iii).

\medskip
\begin{lemma} \label{INVPOLYNOM}
{\rm (i)} $D_{0,p}\left(\mathbb C[z]_n\right) 
\subset \mathbb C[z]_n$
для любых целого $n\ge 0$,
многочлена $p\in \mathcal D(g_0)$ такого, что $p(0)=1$ и $n\ge {\rm deg}(p)-1$.

\noindent
{\rm (ii)} $D_{0, g_0}\left(g_0\mathbb C^-_{\Upsilon}(z)\right)\subset g_0\mathbb C^-_{\Upsilon}(z)$
для любого конечного множества $\Lambda\subset\mathbb C\backslash\Omega$,
любых $n_\lambda\in\mathbb N$ (здесь $\Upsilon:=\{(\lambda, n_\lambda)\,|\,\lambda\in\Lambda\}$).
\end{lemma}

\begin{proof}
(i): Если $f\in \mathbb C[z]_n$, то $D_{0,p}(f)(t)=\frac{f(t)-p(t)f(0)}{t}$ также многочлен
степени не выше $n$.

(ii): Вытекает из замечания \ref{drobi}, 4), поскольку
для любых $\lambda\in\mathbb C\backslash\Omega$,
$k\in\mathbb N$ выполняется равенство $D_{0,g_0}(g_0 q_{\lambda, k})=g_0D_0(q_{\lambda, k})$.
\end{proof}

\begin{remark} \label{GOSHA}
{\rm
Термин {\it обобщенный обратный сдвиг} был введен Ж. Годфруа и Дж. Шапиро
\cite[\S~3]{GOSHA}. Так в \cite{GOSHA} назван линейный непрерывный оператор в банаховом
пространстве $X$, ядро которого одномерно, а объединение ядер всех его
целых неотрицательных степеней плотно в $X$.
Для $n\in\mathbb N$ ядро $D_{0,g_0}^n$ в (ненормируемом) пространстве Фреше $H(\Omega)$ совпадает
с $g_0\mathbb C[z]_{n-1}$ (см. замечание \ref{drobi},\,1)).
Поэтому $\bigcup\limits_{n\ge 0}{\rm Ker}(D_{0,g_0}^n)=g_0\mathbb C[z]$
и замыканием этого множества в $H(\Omega)$ является $g_0 H(\Omega)$.
Значит, второе условие в определении обобщенного обратного
сдвига в смысле \cite{GOSHA} для названного нами так же оператора
$D_{0, g_0}$ выполняется тогда и только тогда, когда $g_0$ не имеет нулей в $\Omega$.
}
\end{remark}

\section{Основные результаты}

Ситуация, когда $g_0$ не имеет нулей в $\Omega$, сводится к случаю $g_0\equiv 1$ с
помощью следующего простого соображения.

\begin{lemma} \label{SVEDENIE} Пусть функция $g_0$ не имеет нулей в $\Omega$. Множество
$S\subset H(\Omega)$ является собственным замкнутым $D_{0,g_0}$-инвариантным подпространством
$H(\Omega)$ тогда и только тогда, когда $\frac{1}{g_0}S$ является собственным
замкнутым $D_0$-инвариантным подпространством $H(\Omega)$.
\end{lemma}

\begin{proof}
Данное утверждение вытекает из равенства
$D_{0,g_0}(f)=g_0 D_0\left(f/g_0\right)$, $f\in H(\Omega)$,
и того, что отображение $f\mapsto \ g_0 f$, $f\in H(\Omega)$, является линейным
топологическим изоморфизмом $H(\Omega)$ на себя.
\end{proof}

Для подмножеств $M_1, M_2$ линейного пространства
$L$ через $M_1+M_2$, как обычно, обозначим их сумму: $M_1+M_2:=\{a+b\,|\, a\in M_1,\, b\in M_2\}$.

\subsection{Случай области, отличной от комплексной плоскости}

\begin{theorem} \label{NOZERO} Пусть $\Omega$ --- односвязная область в $\mathbb C$,
содержащая начало, $\Omega\ne\mathbb C$ и функция $g_0$ не имеет нулей в $\Omega$.

\begin{itemize}
\item[(i)] Для любого $n\in\mathbb N\cup\{-\infty, 0\}$,  конечного или пустого
множества $\Lambda\subset\mathbb C\backslash\Omega$,
кратного многообразия $\Upsilon=\{(\lambda, n_\lambda)\,|\, \lambda\in\Lambda\}$
в $\mathbb C\backslash\Omega$ множество $g_0(\mathbb C[z]_n + \mathbb C_\Upsilon^-(z))$
является замкнутым $D_{0,g_0}$-инвариант\-ным подпространством $H(\Omega)$.
При этом оно является собственным тогда и только тогда, когда $n\ne-\infty$ или $\Lambda$ непусто.
\item[(ii)] Любое собственное замкнутое  $D_{0,g_0}$-инвариантное
подпространство $S$ пространства $H(\Omega)$ содержится в $g_0\mathbb C(z)$ и выполняется равенство
$$
S=g_0\left(\mathbb C[z]_{n(S)} + \mathbb
C_{\Upsilon(S)}^-(z)\right).
$$
\end{itemize}
\end{theorem}

\begin{proof}
Воспользуемся леммой \ref{SVEDENIE}, сводящей задачу к рассмотрению функции ${\bf 1}$
(вместо $g_0$).

(i): $D_{0}$-инвариантность собственных подпространств $\mathbb C[z]_n$ и
$\mathbb C_\Upsilon^-(z)$ доказана в лемме \ref{INVPOLYNOM}. Их замкнутость следует из их конечномерности.

(ii): Пусть $S$ --- собственное замкнутое $D_{0}$-инвариантное подпространство
$H(\Omega)$. Поскольку функция ${\bf 1}$ не имеет нулей, то $S\subset\mathbb C(z)$.
Пусть $n=n(S)$ --- максимальная степень многочленов $r(f)$, $f\in S\backslash\{0\}$,
в представлении (\ref{PREDSTAVLENIE}),
а $\Upsilon=\Upsilon(S)=\{(\lambda, n_\lambda(S))\,|\,\lambda\in\mathcal P(S)\}$ ---
сингулярное кратное многообразие $S$. Покажем, что
$S=\mathbb C[z]_n+\mathbb C^-_{\Upsilon}(z)$.
Ясно, что $S\subset\mathbb C[z]_n+\mathbb C^-_{\Upsilon}(z)$.
По лемме \ref{POLDROBI}\,(i) \,
$\mathbb C[z]_n\subset S$.

Зафиксируем теперь $\lambda\in\mathcal P(S)$, если $\mathcal P(S)\ne\emptyset$.
Поскольку $\mathbb C[z]_n\subset S$, то подпространство $S$ содержит
некоторую функцию $f=\sum\limits_{j=1}^{n_\lambda} \alpha_j q_{\lambda,j}
+ \sum\limits_{\nu\in\mathcal P(S), \nu\ne\lambda}\sum\limits_{k=1}^{n_\nu}\beta_{\nu, k}q_{\nu, k}$
($\alpha_j, \beta_{\nu,k}\in\mathbb C$), \,$\alpha_{n_\lambda}\ne 0$.
Из замечаний \ref{drobi}, 3), 5), 6) вытекает, что
все функции $q_{\lambda, k}$, $\lambda\in\mathcal P(S)$, $1\le k\le n_\lambda$,
принадлежат $S$. Так как система
$\left\{q_{\lambda, k}\,|\,\lambda\in\mathcal P(S), 1\le k\le n_\lambda\right\}$
образует базис в пространстве $\mathbb C^-_{\Upsilon}(z)$, то $\mathbb C^-_{\Upsilon}(z)\subset S$.
Таким образом, $S=\mathbb C[z]_n+\mathbb C^-_{\Upsilon}(z)$.
\end{proof}

\medskip
\begin{theorem} \label{WITHZERO}
Пусть $\Omega$ --- односвязная область в $\mathbb C$, содержащая начало,
$\Omega\ne\mathbb C$ и функция $g_0$ имеет нули в $\Omega$.
\begin{itemize}
\item[(i)] Для любого непустого кратного многообразия $W\prec W(g_0)$ в $\Omega$ множество $S(W)$
является собственным замкнутым $D_{0, g_0}$-инвариан\-тным подпространством
$H(\Omega)$.
\item[(ii)]
Для любого многочлена $p\in\mathcal D(g_0)$,
любых целого $n\ge 0$ такого,
что $n\ge {\rm deg}(p)-1$, или $n=-\infty$,
конечного или пустого кратного многообразия $\Upsilon=\{(\lambda, n_\lambda)\, |\,\lambda\in\Lambda\}$
в $\mathbb C\backslash\Omega$ множество
$\frac{g_0}{p}\mathbb C[z]_n + g_0\mathbb C^-_\Upsilon(z)$
является замкнутым $D_{0,g_0}$-инвариан\-тным подпространством $H(\Omega)$.

При этом оно собственное тогда и только тогда, когда $n\ne-\infty$ или $\Upsilon$ непусто.
\item[(iii)]
Для любого собственного замкнутого $D_{0, g_0}$-инвариантного подпространства $S$
пространства $H(\Omega)$
либо $S=S(W(S))$, либо $S\subset g_0\mathbb C(z)$ и
выполняется равенство
$S=\frac{g_0}{p_S}\mathbb C[z]_{n(S)} + g_0\mathbb C^-_{\Upsilon(S)}(z)$.
\end{itemize}
\end{theorem}

\begin{proof}
(i): Множество $S(W)$ является собственным замкнутым подпространством $H(\Omega)$.
Пусть $W=\{(\lambda_k, m_k)\}$.
Возьмем $f\in S(W)$. Зафиксируем $k$. Тогда $f(t)=(t-\lambda_k)^{m_k}f_k(t)$,
$g_0(t)=(t-\lambda_k)^{m_k}h(t)$, $t\in\Omega$, где $f_k, h\in H(\Omega)$.
Поэтому для функции
$$
D_{0,g_0}(f)(t)=(t-\lambda_k)^{m_k}\frac{f_1(t) - h(t)f(0)}{t}
$$
$\lambda_k$ --- нуль кратности не меньше $m_k$. Значит, $D_{0,g_0}(f)\in S(W)$.

(ii): Пусть $S$ такое, как в (ii). Тогда $S$ --- замкнутое подпространство
$H(\Omega)$. Инвариантность $S$ относительно $D_{0,g_0}$ следует из леммы \ref{INVPOLYNOM}.
При этом подпространство $\frac{g_0}{p}\mathbb C[z]_n+g_0\mathbb C_\Upsilon^-(z)$ не является собственным
тогда и только тогда, когда оно тривиально. Это равносильно тому, что $n=-\infty$ и $\Upsilon$ пусто.

(iii): Предположим, что $S\backslash\left(g_0\mathbb C(z)\right)\ne\emptyset$.
По лемме \ref{MNOGOOBRASIE} множество $Z(S)$ общих нулей функций из
$S$ в $\Omega$ непусто и $W\prec W(g_0)$. Отсюда следует, что $0\notin Z(S)$.
Пусть $v$ --- $g_0$-экстремальная функция $S$.
По теореме Вейерштрасса \cite[теорема 3.3.1]{BEG} существует голоморфная в $\Omega$ функция $u$
с нулевым многообразием $W(S)$ и такая, что $u(0)=1$.
Тогда $g_0/u, v/u\in H(\Omega)$ и $S(W(S))=u H(\Omega)$.
Вследствие свойства расщепляемости (замечание \ref{RAS1})
\begin{equation} \label{RAS2}
D_{0,g_0}^n(v)=u D_{0,g_0/u}^n\left(\frac{v}{u}\right), \,\, n\ge 0.
\end{equation}

Рассмотрим случай, когда $v\notin g_0\mathbb C(z)$. Тогда $\frac{v}{u}\notin\frac{g_0}{u}\mathbb C(z)$.
При этом $v/u$ и $g_0/u$ не имеют общих нулей в $\Omega$.
По теореме \ref{cycl} функция $v/u$ --- циклический вектор $D_{0,g_0/u}$ в $H(\Omega)$.
Следовательно, в силу (\ref{RAS2}), $u H(\Omega)\subset S$, а значит, $S= S(W(S))$.

Если же $v\in g_0\mathbb C(z)$, то возьмем $f\in S\backslash\left(g_0\mathbb C(z)\right)$.
По лемме \ref{LEMMAEXT} найдется $\alpha\in\mathbb C\backslash\{0\}$,
для которого $v_0=f+\alpha v$ --- тоже $g_0$-экстремальная функция $S$.
При этом $v_0\notin g_0\mathbb C(z)$. Отсюда, как выше
(для функции $v_0$ вместо $v$), следует, что $S=S(W(S))$.

Пусть $S\subset g_0\mathbb C(z)$.
Применяя рассуждения, аналогичные приведенным при доказательстве теоремы \ref{NOZERO}\,(ii)
(они используют лемму \ref{POLDROBI}\,(i), замечания \ref{drobi}, 3), 5), 6) и равенство $D_{0,g_0}(g_0 h)=g_0 D_0(h)$,
$h\in H(\Omega)$), получим,
что $S=\frac{g_0}{p}\mathbb C[z]_{n(S)} + g_0\mathbb C^-_{\Upsilon(S)}(z)$.
\end{proof}

\subsection{Случай $\Omega=\mathbb C$} Выделим отдельно случай, когда
$\Omega$ совпадает со всей комплексной плоскостью. В этой ситуации
$\mathbb C\backslash\Omega=\emptyset$. Поэтому вследствие доказательств
теорем 2,\,3 имеют место следующие утверждения.

\begin{theorem} \label{PLANE}
Пусть $g_0$ --- целая функция такая, что $g_0(0)=1$.

\noindent
{\rm (I)} Предположим, что $g_0$ не имеет нулей в $\mathbb C$.

\noindent
\begin{itemize}
\item[(i)] Для любого целого $n\ge 0$ множество $g_0\mathbb C[z]_n$
является собственным замкнутым $D_{0,g_0}$-инвари\-ант\-ным подпространством $H(\mathbb C)$.
\item[(ii)] Любое собственное замкнутое  $D_{0,g_0}$-инвариантное
подпространство $S$ пространства $H(\mathbb C)$  вложено в $g_0\mathbb C[z]$,
и выполняется равенство $S=g_0\mathbb C[z]_{n(S)}$.
\end{itemize}

\noindent
{\rm (II)} Предположим, что $g_0$ имеет нули в $\mathbb C$.
\begin{itemize}
\item[(iii)] Для любого непустого кратного многообразия $W\prec W(g_0)$ в $\mathbb C$ множество $S(W)$
является собственным замкнутым $D_{0, g_0}$-инвариан\-тным подпространством
$H(\mathbb C)$.
\item[(iv)]
Для любых многочлена $p\in\mathcal D(g_0)$, целого $n\ge 0$ такого,
что $n\ge {\rm deg}(p)-1$, множество
$\frac{g_0}{p}\mathbb C[z]_n$
является собственным замкнутым $D_{0,g_0}$-инвариантным подпространством $H(\mathbb C)$.
\item[(v)]
Для любого собственного замкнутого $D_{0, g_0}$-инвариантного подпространства $S$
пространства $H(\mathbb C)$ либо $S=S(W(S))$, либо
$S\subset g_0\mathbb C[z]$ и выполняется равенство $S=\frac{g_0}{p_S}\mathbb C[z]_{n(S)}$.
\end{itemize}
\end{theorem}

Из теорем 2--4 вытекает

\begin{corollary} Пусть $\Omega$ --- односвязная область в $\mathbb C$, содержащая начало.
Оператор $D_{0,g_0}$ является одноклеточным в $H(\Omega)$ тогда
и только тогда, когда $\Omega=\mathbb C$ и $g_0$ не имеет нулей в $\mathbb C$.

Если это так, то собственными замкнутыми $D_{0,g_0}$-инвариантными подпространствами
$H(\Omega)$ являются $g_0\mathbb C[z]_n$, $n\ge 0$, и только они.
\end{corollary}




\begin{flushleft}
О.~А.~Иванова

Южный федеральный университет, \\
Институт математики, механики и компьютерных наук им.
И.~И.~Воровича, \\
Ростов-на-Дону, Россия

E-mail: neo$_{-}$ivolga@mail.ru
\end{flushleft}

\begin{flushleft}
С.~Н.~Мелихов

Южный федеральный университет, \\
Институт математики, механики и компьютерных наук им.
И.~И.~Воровича,\\
Ростов-на-Дону, Россия;\\
Южный математический институт Владикавказского научного центра
РАН, \\
Владикавказ, Россия

E-mail: snmelihov@yandex.ru, melih@math.rsu.ru

\end{flushleft}

\begin{flushleft}
Ю.~Н.~Мелихов

Военная академия ВКО им.~Г.~К.~Жукова, \\
Тверь, Россия

E-mail: melikhow@mail.ru

\end{flushleft}

\end{document}